\documentclass[12pt]{article} 

\usepackage{color}
\usepackage[margin=1in]{geometry} 
\usepackage{graphicx} 
\usepackage{amsmath} 
\usepackage{amsfonts} 
\usepackage{amsthm} 
\usepackage{amssymb}
\usepackage{mathrsfs}
\usepackage{mathtools}
\usepackage{enumitem}
\usepackage{hyperref}
\usepackage{dsfont}
\theoremstyle{theorem}
\newtheorem{thm}{Theorem}[section]
\newtheorem{lem}[thm]{Lemma}
\newtheorem{hyp}[thm]{Hypothesis}

\theoremstyle{definition}

\newtheorem{remark}[thm]{Remark}

\newtheorem{counterexample}[thm]{Counterexample}
\newcommand{\de}{\mathrm{d}}
\newcommand{\R}{\mathbb{R}}
\newcommand{\N}{\mathbb{N}}
\DeclareMathAlphabet{\mathpzc}{OT1}{pzc}{m}{it}
\providecommand{\keywords}[1]{\textbf{\textit{Keywords: }}#1}

\title{A Stochastic Gronwall Lemma and Well-Posedness of Path-Dependent SDEs Driven by Martingale Noise}

\author{Sima Mehri\footnotemark[1]\ \footnotemark[2]\ \footnotemark[3] \and Michael Scheutzow\footnotemark[1]}

\begin{document}

\maketitle
\renewcommand{\thefootnote}{\fnsymbol{footnote}}
\footnotetext[1]{Institut f\"ur Mathematik, Technische Universit\"at Berlin, D-10623 Berlin, Germany}
\footnotetext[2]{Department of Mathematical Sciences, Sharif University of Technology, Tehran, Iran}
\footnotetext[3]{The work of this author was supported by the Hilda Geiringer Scholarship awarded by the Berlin Mathematical School}
\renewcommand{\thefootnote}{\arabic{footnote}}

\begin{abstract}
We show   existence and uniqueness of solutions of stochastic path-dependent differential equations driven by c\`adl\`ag martingale noise under joint local monotonicity and coercivity assumptions on the coefficients with a bound in terms of the  supremum norm. In this set-up the usual proof using the ordinary Gronwall
  lemma together with the  Burkholder-Davis-Gundy inequality seems impossible. In order to solve this problem, we prove  a new and quite general stochastic Gronwall lemma for c\`adl\`ag martingales using Lenglart's inequality. 
\end{abstract}

\keywords{stochastic Gronwall lemma, functional stochastic differential equations, path-dependent stochastic differential equations, martingale inequality, monotone coefficients, Lenglart inequality}


\bigskip 
\section{Introduction}
Fix $\tau>0$ and let $(\Omega,\mathcal{F},(\mathcal{F}_t)_{t\geq 0}, \mathbb{P})$ be a {\em normal} filtered  probability space, i.e.~the space is complete and  satisfies the usual conditions. Consider the following stochastic delay differential equation in $\mathbb{R}^d$:
\begin{equation}\label{Chapter1-mao-equ}
\begin{dcases}\de X(t)=f(t,\omega, X_{t-\tau:t})\de t+\int_U g(t,\omega, X_{t-\tau:t},\xi)\tilde{M}(\de t,\de \xi),\\X(t)=z(t), \quad t\in[-\tau,0],\end{dcases}
\end{equation}
where $X_{t-\tau:t}(s)=X(t+s), s\in[-\tau,0]$ and $z\in L^2(\Omega,\mathcal{F}_0,\mathbb{P}; \text{C\`adl\`ag}([-\tau,0],\mathbb{R}^d))$. 

We will state precise assumptions on $\tilde M$ later. At the moment, assume that $U=U_1\sqcup U_2$, where $U_1$ is a
finite or infinite subset of $\N$ and the integral over $U_1$ is a sum, where $\tilde M_t(i)$, $i \in U_1$ are independent
Wiener processes and the remaining integral over $U_2$ is with respect to compensated Poisson noise which is independent of
the Wiener processes. If $U_2=\emptyset$, then we speak of {\em Wiener} or {\em diffusive} noise, otherwise of
{\em jump diffusive} noise. In the diffusive case, several authors established existence and uniqueness of solutions of  \eqref{Chapter1-mao-equ} under
various conditions on the coefficients (e.g.~\cite[Theorem 5.2.5]{mao2007stochastic} under local Lipschitz and linear growth
assumptions on $f$ and $g$ and \cite{von2010existence} under a  one-sided local Lipschitz and a suitable growth condition).
Under similar conditions, \cite{wei2007existence} and  \cite{ren2008remarks} show existence and uniqueness even for
equations with infinite delay and \cite{MS03} (see also \cite{CS13}) proved not only existence and uniqueness but also pathwise continuous dependence
of the solution on the initial condition in case $g$ does not depend on the past (otherwise it is known that pathwise 
continuous dependence on the initial condition does not hold in general, see \cite{MS97}).
Existence and uniqueness results in the jump diffusive case under a local Lipschitz and linear growth condition (even with
additional Markovian switching) were obtained in  \cite{zhu2017razumikhin}.

In both the existence and the uniqueness proof one typically encounters  the following inequality for some non-negative adapted process  $Z$,
\begin{equation}\label{Chapter1-Z-ineq}
Z(t)\leq K\int_0^t Z^\star(s)\de s+M(t)+H(t),
\end{equation}
where $Z^\star (s)=\sup_{u\in[0,s]} Z(u)$, $M$ is a local martingale (depending on the function $g$ in the equation), the process $H(t), t\geq 0$ is non-decreasing adapted, and $K>0$ is a constant.
In order to apply  Gronwall's lemma, the expression inside the integral should be the same as the expression on the left side of the inequality. Taking the supremum on both sides of \eqref{Chapter1-Z-ineq}
and then taking expectations, an upper bound for $\mathbb{E}M^\star(t)$ in terms of the process $Z$ is required. Under a local one-sided Lipschitz condition of the form
\begin{equation}\label{Chapter1-local monotone scheutzow}
\begin{gathered}\text{For all compact subset } \mathcal{C}\subset C([-\tau,0],\mathbb{R}^d)\text{ there exists } L_{\mathcal{C}}>0\text{ and } \\ \tau_{\mathcal{C}}\in (-\tau,0]\text{ such that }\forall x,y\in  \mathcal{C}\text{ with } x(s)=y(s)\ \forall s\in [-\tau,-\tau_{\mathcal{C}}]\\ 2\left\langle x(0)-y(0), f(x)-f(y)\right\rangle+\left\lvert g(x)-g(y)\right\rvert^2\leq L_{\mathcal{C}} \sup_{s\in [-\tau,0]}\left\lvert x(s)-y(s)\right\rvert^2,\end{gathered}
\end{equation}
as in \cite{von2010existence}, controls with respect to the  supremum norm on  $g$ are not separated from $f$ and it therefore seems impossible to use the Burkholder-Davis-Gundy inequality to obtain an
upper bound for $\mathbb{E}M^\star(t)$ in this case. 

The paper \cite{von2010existence} dealt with this problem by proving the following stochastic Gronwall's inequality for the above mentioned process $Z$ and for $p\in (0,1)$ and $\alpha>\frac{1+p}{1-p}$:
	\[\mathbb{E}\left[( Z^\star(T))^p\right]\leq c_1 e^{c_2KT}\left(\mathbb{E}\left[H(T)^\alpha\right]\right)^{p/\alpha}, \quad\forall T\geq 0.\]
        Here $c_1$ and $c_2$ are two constants that only depend on $p$ and $\alpha$ and $Z$, $H$, and $M$ are assumed to have continuous paths (in addition to the properties stated above).

        One can find another type of stochastic Gronwall lemma in the literature where $Z^\star(s)K\de s$ in the assumption is replaced by $Z(s^-)\de A(s)$ for an adapted non-decreasing stochastic process $A$
        (see \cite{scheutzow2013stochastic} for continuous processes, \cite{zhang2018singular} for c\`adl\`ag processes and \cite{kruse2018discrete} for discrete time processes).

Whenever the supremum norm in condition \eqref{Chapter1-local monotone scheutzow} is replaced  by a real-valued continuous linear operator, say $\lambda$, on $\text{C\`adl\`ag}([-\tau,0],\mathbb{R})$, then there is no problem using the ordinary Gronwall's lemma. In \cite{mehri2018propagation}, we have  stated the well-posedness of equation \eqref{Chapter1-mao-equ} driven by jump diffusion under the local monotonicity assumption,
\begin{equation}\label{Chapter1-plus monotone L2}
\begin{gathered}\forall R>0,\ \exists L_R\in L^1_{loc}(\mathbb{R}_{\geq 0},\mathbb{R}_{\geq 0}),\ \forall x,y\in  \text{C\`adl\`ag}([-\tau,0],\mathbb{R}^d)\\\text{ with }\sup_{s\in[-\tau,0]}\left\lvert x(s)\right\rvert ,\sup_{s\in[-\tau,0]}\left\lvert y(s)\right\rvert<R:\\2\left\langle x(0^-)-y(0^-), f(t,\omega,x)-f(t,\omega,y)\right\rangle+ \int_U \left\lvert g(t,\omega,x,\xi)-g(t,\omega,y,\xi)\right\rvert^2\nu_t(\de \xi)\\\leq L_R(t) \lambda\left(\left\lvert x(\cdot)-y(\cdot)\right\rvert^2\right), \end{gathered}
\end{equation}
and coercivity assumption,
\begin{equation}\label{Chapter1-plus coercive L2}
\begin{gathered}\exists K\in L^1_{loc}(\mathbb{R}_{\geq 0},\mathbb{R}_{\geq 0}),\ \forall x\in  \text{C\`adl\`ag}([-\tau,0],\mathbb{R}^d):\\2\left\langle x(0^-), f(t,\omega,x)\right\rangle+ \int_U \left\lvert g(t,\omega,x,\xi)\right\rvert^2\nu_t(\de \xi)\leq K(t) \lambda\left(1+\left\lvert x(\cdot)\right\rvert^2\right) \end{gathered}
\end{equation}
without using a stochastic Gronwall lemma.

In this paper, we study existence and uniqueness of equation
\begin{equation}\label{Chapter1-sima-equ}
\begin{dcases}\de X(t)=f(t,\omega, X)\,\de t+\int_U g(t,\omega, X,\xi)\tilde{M}(\de t,\de \xi),\\X(t)=z(t),\quad t\in[-\tau,0],\end{dcases}
\end{equation}
under weaker conditions than those stated above. In particular, $\tilde M$ will be a rather general martingale measure, and $f$ and $g$ satisfy weaker conditions than   \eqref{Chapter1-plus monotone L2}
and \eqref{Chapter1-plus coercive L2}, namely the right hand sides are replaced by the supremum norm. We will state precise conditions later.

\section{Stochastic Gronwall Lemma}\label{gronsec}
Throughout this section, we will assume that $(\Omega,\mathcal{F},\mathbb{P})$ is a probability space with normal filtration $(\mathcal{F}_t)_{t\ge 0}$.
We will use the following lemma which is essentially \cite[Th\'eor\`eme I \& Corollaire II]{lenglart1977relation} with a slightly better constant $c_p$ and slightly weaker
assumptions.  Note that \cite[Proposition IV.4.7 \& Exercise IV.4.30]{revuz1999continuous} states a similar result for the case of continuous $G$.
\begin{lem}\label{Chapter2-Martingale ineq}
	Let $X$ be a non-negative adapted right-continuous process and let $G$ be a non-negative right-continuous non-decreasing predictable process such that $\mathbb{E}[X(\tau)\vert \mathcal{F}_0]\leq \mathbb{E}[G(\tau)\vert \mathcal{F}_0]\le \infty$ for
	any bounded stopping time $\tau$. Then
	\begin{itemize}
		\item[(i)] $\displaystyle\forall c,d>0$, 
		\[\mathbb{P}\left(\sup_{t\geq 0}X(t)>c\,\Big\vert\mathcal{F}_0\right)\leq \frac{1}{c}\mathbb{E} \left[\sup_{t\geq 0}G(t)\wedge d\,\Big\vert\mathcal{F}_0\right]+\mathbb{P}\left(\sup_{t\geq 0}G(t)\geq d\,\Big\vert\mathcal{F}_0\right).\]
		\item[(ii)] For all  $p\in(0,1)$,
		\[\mathbb{E}\left[\left(\sup_{t\geq 0}X(t)\right)^p\Big\vert \mathcal{F}_0 \right]\leq c_p\mathbb{E}\left[\left(\sup_{t\geq 0}G(t)\right)^p\Big\vert \mathcal{F}_0\right],\]
		where $c_p:=\frac{p^{-p}}{1-p}$.
	\end{itemize}
\end{lem}
For the proof of this lemma, recall that a {\em predictable stopping time} is a map $\tau: \Omega \to [0,\infty]$ for which there exists an increasing sequence $(\tau_n)_{n\in\mathbb{N}}$ of stopping times (called {\em announcing sequence} for $\tau$) with the properties
	\begin{itemize}
		\item[(a)]$\lim_{n\to\infty}\tau_n(\omega)=\tau(\omega),\forall \omega\in \Omega$,
		\item[(b)]$\tau_n(\omega)<\tau(\omega), \forall \omega\in \{\tau>0\}$
	\end{itemize}
        (see \cite[p56]{D72}). For $A \subset [0,\infty)\times \Omega$, let $T_A(\omega):=\inf\{t\ge 0 :(t,\omega)\in A\}$ be the first hitting time of $A$. If $A$ is predictable and
        $\{(t,\omega):T_A(\omega)=t\}\subset A$, then $T_A$ is a predictable stopping time (\cite[p74]{D72}).
\begin{proof}[Proof of Part (i)] This is essentially Theorem I in \cite{lenglart1977relation} with two small modifications: both the assumption and the conclusion
	in \cite{lenglart1977relation} are formulated for expected values rather than conditional expectations and 
	\cite{lenglart1977relation} assumes that $G(0)=0$ almost surely which we do not assume. Both generalizations are easy to see but for the convenience of the reader we provide a proof.
	
	Let $\tilde{\tau}_d:=\inf\{t\geq 0: G(t)\geq d\}$ and $\tau_c:=\inf\{t\geq 0: X(t)\geq c\}$. Since $G$ is a predictable process, $\tilde{\tau}_d$ is the first hitting time of the
        predictable set $A=\{(t,\omega):G(t,\omega)\ge d\}$ and hence is a predictable stopping time since $\{(t,\omega):\tilde{\tau}_d (\omega)=t\}\subset A$. 
        Therefore, there exists a sequence of stopping times $\tilde{\tau}_d^n, n\in \mathbb{N}$  such that $\tilde{\tau}_d^n\uparrow \tilde{\tau}_d$ as $n\uparrow\infty$ and $\tilde{\tau}_d^n<\tilde{\tau}_d$ for all $n\in\mathbb{N}$ on $\{\tilde{\tau}_d>0\}=\{G(0)<d\}$. Then for $T>0$,
	\begin{align*}
	\MoveEqLeft[3]\mathbb{P}\left(\sup_{t\in [0,T]}X(t)>c\,\Big\vert\mathcal{F}_0\right)\\&=\mathbb{P}\left(\sup_{t\in [0,T]}X(t)>c, G(T)<d\,\Big\vert\mathcal{F}_0\right)+\mathbb{P}\left(\sup_{t\in [0,T]}X(t)>c, G(T)\geq d\,\Big\vert\mathcal{F}_0\right)
          \\
   &\leq \mathbb{P}\left(\big\{\mathbf{1}_{\{G(0)<d\}}X(T\wedge  \tau_c )\geq c\big\} \cap \big\{\tilde \tau_d > T\big\}\,\Big\vert\mathcal{F}_0\right)+\mathbb{P}\left(G(T)\geq d\,\vert\mathcal{F}_0\right)\\
   &= \lim_{n \to \infty}\mathbb{P}\left(\big\{\mathbf{1}_{\{G(0)<d\}}X(T\wedge  \tau_c )\geq c\big\} \cap \big\{\tilde \tau^n_d > T\big\}\,\Big\vert\mathcal{F}_0\right)+\mathbb{P}\left(G(T)\geq d\,\vert\mathcal{F}_0\right)\\
   &= \lim_{n \to \infty}\mathbb{P}\left(\big\{\mathbf{1}_{\{G(0)<d\}}X(T\wedge \tilde \tau^n_d\wedge \tau_c )\geq c\big\} \cap \big\{\tilde \tau^n_d > T\big\}\,\Big\vert\mathcal{F}_0\right)+\mathbb{P}\left(G(T)\geq d\,\vert\mathcal{F}_0\right)\\
   &\leq  \lim_{n \to \infty}\mathbb{P}\left(\big\{\mathbf{1}_{\{G(0)<d\}}X(T\wedge \tilde \tau^n_d\wedge \tau_c )\geq c\big\}\,\Big\vert\mathcal{F}_0\right)+\mathbb{P}\left(G(T)\geq d\,\vert\mathcal{F}_0\right)
	\\&
	\leq \frac{1}{c}\lim_{n\to\infty}\mathbb{E} \left[\mathbf{1}_{\{G(0)<d\}}G(T\wedge \tilde{\tau}_d^n\wedge \tau_c )\Big\vert\mathcal{F}_0\right]+\mathbb{P}\left(G(T)\geq d\,\vert\mathcal{F}_0\right)\\&\leq \frac{1}{c}\mathbb{E}[G(T)\wedge d|\mathcal{F}_0]+\mathbb{P}\left(G(T)\geq d\,\vert\mathcal{F}_0\right).
	\end{align*}
%
%
%
Taking the limit $T\to +\infty$ the result follows.
\end{proof}  
\begin{proof}[Proof of Part (ii)]
	Using part (i), we have, for $\lambda>0$,
	{\small
		\begin{align*}
		\MoveEqLeft[0]\mathbb{E}\left[\left(\sup_{t\geq 0}X(t)\right)^p\Big\vert \mathcal{F}_0 \right]= \int_0^{+\infty}\mathbb{P}\left(\sup_{t\geq 0}X(t)>c^{1/p}\Big\vert\mathcal{F}_0\right)\de c\\&\leq \int_{0}^{+\infty}\left\lbrace\frac{1}{c^{1/p}} \mathbb{E} \left[\sup_{t\geq 0}G(t)\wedge \lambda c^{1/p}\Big\vert\mathcal{F}_0\right]+\mathbb{P}\left(\sup_{t\geq 0}G(t)\geq \lambda c^{1/p}\Big\vert\mathcal{F}_0\right)\right\rbrace\de c\\&=\mathbb{E}\left[\int_0^{(\sup_{t\geq 0}G(t)/\lambda)^p}\lambda\de c+\int_{(\sup_{t\geq 0}G(t)/\lambda)^p}^{+\infty}\frac{\sup_{t\geq 0}G(t)}{c^{1/p}}\de c\Big\vert\mathcal{F}_0\right]+\lambda^{-p}\mathbb{E}\left[\left(\sup_{t\geq 0}G(t)\right)^p\Big\vert\mathcal{F}_0\right]\\&=\left(\frac{1}{1-p}\lambda^{1-p}+\lambda^{-p}\right) \mathbb{E}\left[\left(\sup_{t\geq 0}G(t)\right)^p\Big\vert \mathcal{F}_0\right]
		\end{align*}}
	The minimal value of  $(1-p)^{-1}\lambda^{1-p}+\lambda^{-p}$  is equal to $c_p$ for the minimizer $\lambda=p$.
\end{proof}
\begin{thm}[Stochastic Gronwall lemma]\label{Chapter2-GronwallThm}
	Let $X(t),\, t\geq 0$ be an $(\mathcal{F}_t)_{t\ge 0}$-adapted non-negative right-continuous process.
	Assume that $A:[0,\infty)\to[0,\infty) $ is a deterministic non-decreasing c\`adl\`ag function with $A(0)=0$ and let $H(t),\,t\geq 0$
	be a non-decreasing and c\`adl\`ag adapted process starting from $H(0)\geq 0$. Further, let $M(t),\,t\geq 0$ be  an
	$(\mathcal{F}_t)_{t\ge 0}$- local martingale with $M(0)=0$ and c\`adl\`ag paths.  Assume that for all $t\geq 0$,
	
	\begin{equation}\label{Chapter2-asp-ineq}
	X(t)\leq \int_0^t X^*(u^-)\,\de A(u)+M(t)+H(t),
	\end{equation}
	
	where $X^*(u):=\sup_{r\in[0,u]}X(r)$. Then the following estimates hold for $p\in (0,1)$ and $T>0$.
	\begin{enumerate}[label=(\alph*)]

		\item\label{Chapter2-a}If $\mathbb{E} \big(H(T)^p\big)<\infty$ and $H$ is predictable, then
		\begin{equation}\label{Chapter2-ineq-p-H:predic}
		\mathbb{E}\left[\left(X^*(T)\right)^p\Big\vert\mathcal{F}_0\right]\leq \frac{c_p}{p}\mathbb{E}\left[(H(T))^p\big\vert\mathcal{F}_0\right] \exp \left\lbrace c_p^{1/p}A(T)\right\rbrace.
		\end{equation}
		\item\label{Chapter2-b}If $\mathbb{E} \big(H(T)^p\big)<\infty$ and $M$ has no negative jumps, then
		\begin{equation}\label{Chapter2-ineq-p-M:cont}
		\mathbb{E}\left[\left(X^*(T)\right)^p\Big\vert\mathcal{F}_0\right]\leq \frac{c_p+1}{p}\mathbb{E}\left[(H(T))^p\big\vert\mathcal{F}_0\right] \exp \left\lbrace (c_p+1)^{1/p}A(T)\right\rbrace.
		\end{equation}
		\item\label{Chapter2-c}If $\mathbb{E} H(T)<\infty$, then  
		\begin{equation}\label{Chapter2-ineq-p}
		\displaystyle{\mathbb{E}\left[\left(X^*(T)\right)^p\Big\vert\mathcal{F}_0\right]\leq \frac{c_p}{p}\left(\mathbb{E}\left[ H(T)\big\vert\mathcal{F}_0\right]\right)^p \exp \left\lbrace c_p^{1/p} A(T)\right\rbrace.}
		\end{equation}
	\end{enumerate} 
	Here $c_p=\frac{p^{-p}}{1-p}$.
\end{thm}
\begin{proof} Note that the usual Gronwall lemma and \eqref{Chapter2-asp-ineq} imply that $X$ is almost surely locally bounded since this holds true for $M$ and $H$ (observe that we did not
  assume that $X$ has left limits).
  
	\textbf{Part (a)}
	Let $\sigma_n$, $n\in\mathbb{N}$ be a localizing sequence of stopping times for the local martingale $M$ and define $\tau_n:=\inf\left\lbrace t\geq 0: X(t)>n\right\rbrace\wedge\sigma_n$. Then it holds that
	\begin{equation}\label{Chapter2-stp-asp}
	X(t\wedge\tau_n)\leq \int_0^t X^*((s\wedge \tau_n)^-)\,\de A(s)+M(t\wedge\tau_n)+H(t)\leq \int_0^t X^*(s^-\wedge \tau_n)\,\de A(s)+M(t\wedge\tau_n)+H(t),
	\end{equation}
	$X$ is a nonnegative right-continuous process and 
	\[G_n(t):=\int_0^t X^*(s^-\wedge\tau_n)\,\de A(s)+H(t)\]
	is non-decreasing and  predictable with the property that for every finite stopping time $\tau$,  we have $\mathbb{E}\left[ X(\tau \wedge \tau_n)\vert \mathcal{F}_0\right]\leq \mathbb{E}\left[ G_n(\tau)\vert \mathcal{F}_0\right]\le \infty$. Therefore, using Lemma \ref{Chapter2-Martingale ineq} and Young's inequality, we have, for $\lambda>0$ and $t \ge 0$
	\begin{align*}
	\MoveEqLeft[0.5]\mathbb{E}\left[ (X^*(t\wedge\tau_n))^p\vert \mathcal{F}_0\right]\\&\leq c_p\mathbb{E}\left[\left(\int_0^tX^*(s^-\wedge\tau_n)\,\de A(s) \right)^p+\left(H(t)\right)^p\Big\vert \mathcal{F}_0\right]\\&\leq c_p\mathbb{E}\left[\left(\int_0^t(X^*(s^-\wedge\tau_n))^p\,\de A(s) \right)^p(X^*(t^-\wedge\tau_n))^{p(1-p)}+\left(H(t)\right)^p\Big\vert \mathcal{F}_0\right]
	\\&\leq c_p\mathbb{E}\left[p\lambda^{1-p}\int_0^t(X^*(s^-\wedge\tau_n))^p\,\de A(s)+(1-p)\lambda^{-p}(X^*(t\wedge\tau_n))^p+\left(H(t)\right)^p\Big\vert \mathcal{F}_0\right].
	\end{align*}
	It follows from the first inequality in \eqref{Chapter2-stp-asp} that  $\mathbb{E}\left[ (X^*(T\wedge\tau_n))^p\vert \mathcal{F}_0\right]<\infty$ almost surely. Hence, applying
        the usual Gronwall's lemma to $f(t):=\mathbb{E}\big(X^*(t \wedge \tau_n)^p\big|\mathcal{F}_0\big)$, we get for $\lambda> c_p^{1/p}(1-p)^{1/p}$, 
	\[\mathbb{E}\left[ (X^*(T\wedge\tau_n))^p\vert \mathcal{F}_0\right]\leq \exp\left(\frac{c_pp\lambda^{1-p}A(T)}{1-c_p(1-p)\lambda^{-p}}\right)\frac{c_p\mathbb{E}\left[(H(T))^p\vert \mathcal{F}_0\right]}{1-c_p(1-p)\lambda^{-p}},\]
        so applying Fatou's lemma, we get 
	\begin{align*}
	\mathbb{E}\left[ (X^*(T))^p\vert \mathcal{F}_0\right]&\leq\liminf_{n\to+\infty}\mathbb{E}\left[ (X^*(T\wedge\tau_n))^p\vert \mathcal{F}_0\right]\\&\leq \exp\left(\frac{c_pp\lambda^{1-p}A(T)}{1-c_p(1-p)\lambda^{-p}}\right)\frac{c_p\mathbb{E}\left[(H(T))^p\vert \mathcal{F}_0\right]}{1-c_p(1-p)\lambda^{-p}}
	\end{align*}
	which yields inequality \eqref{Chapter2-ineq-p-H:predic} by taking $\lambda=c_p^{1/p}$. 
	\paragraph*{Part (b)} Let $\sigma_n$, $n\in\mathbb{N}$ be a localizing sequence of stopping times for the continuous local martingale $M$ and define $\tau_n:=\inf\left\lbrace t\geq 0: X(t)>n\right\rbrace\wedge\sigma_n$. Then it holds that
	\begin{equation}\label{Chapter2-stp-asp-2}
	\tilde{G}_n(t):=-\inf_{s\in[0,t]} M(s\wedge\tau_n)\leq \int_0^t X^*((s\wedge \tau_n)^-)\,\de A(s)+H(t),
	\end{equation}
	$M(t\wedge\tau_n)+\tilde G_n(t)$, $t\geq 0$ is a nonnegative continuous process and $\tilde{G}_n$
	is non-decreasing and  predictable with the property that for every bounded stopping time $\tau$, $\mathbb{E}\left[ M(\tau \wedge \tau_n)\vert \mathcal{F}_0\right]\leq \mathbb{E}\left[ \tilde{G}_n(\tau)\vert \mathcal{F}_0\right]$. Therefore using Lemma \ref{Chapter2-Martingale ineq}, we have
	\begin{align}
	\mathbb{E}\left[\left(\sup_{s\in [0,t]}M(s\wedge\tau_n)\right)^p\Big\vert \mathcal{F}_0\right]\leq c_p\mathbb{E}\left[\left(\int_0^t X^*((s\wedge\tau_n)^-)\,\de A(s)+H(t)\right)^p\Big\vert \mathcal{F}_0\right].
	\end{align}
	Using inequality \eqref{Chapter2-stp-asp}, we get
	\begin{align*}
	\mathbb{E}\left[ (X^*(t\wedge\tau_n))^p\vert \mathcal{F}_0\right]&\leq  (c_p+1)\mathbb{E}\left[\left(\int_0^t X^*((s\wedge\tau_n)^-)\,\de A(s)+H(t)\right)^p\Big\vert \mathcal{F}_0\right].
	\end{align*}
	The rest of the proof is similar to the proof of part \ref{Chapter2-a}.
	\paragraph*{Part (c)} Now we prove the inequality for general $H$. Defining the new local martingale 
	\[\tilde{M}(t):=M(t)+\mathbb{E}\left[H(T)\big\vert \mathcal{F}_t\right]-\mathbb{E}\left[ H(T)\big\vert\mathcal{F}_0\right]\]
	(where we take a c\`adl\`ag modification of $t \mapsto \mathbb{E}\left[H(T)\big\vert \mathcal{F}_t\right]$) and the
	predictable process $\tilde{H}(t):=\mathbb{E}\left[ H(T)\big\vert\mathcal{F}_0\right]$, 
	we have 
	\[X(t)\leq \int_0^t X^*(u^-)\,\de A(u)+\tilde{M}(t)+\tilde{H}(t),\]
	since $\mathbb{E}\left[H(T)\big\vert \mathcal{F}_t\right]\geq H(t)$. Thus the result follows from part \ref{Chapter2-a}.
\end{proof}
\begin{remark}
	Lemma 5.4 in \cite{scheutzow2013stochastic} states a stochastic Gronwall inequality in the case of continuous $M, X, H$
	which is less general than part \ref{Chapter2-b} in Theorem \ref{Chapter2-GronwallThm}. In addition, the proof of
	\cite[Lemma 5.4]{scheutzow2013stochastic}
	contains a gap since the processes $X_i$ defined there can be negative outside of $\Omega_i$.  
\end{remark}
\begin{counterexample}
	Under the assumptions of Theorem \ref{Chapter2-GronwallThm}, for $p,\alpha\in (0,1)$, the inequality 
	\[\mathbb{E}\left[\left(X^*(T)\right)^p\Big\vert\mathcal{F}_0\right]\leq c_{1,p,\alpha}\left(\mathbb{E}\left[(H(T))^\alpha\big\vert\mathcal{F}_0\right]\right)^{p/\alpha} \exp \left\lbrace c_{2,p,\alpha}A(T)\right\rbrace\]  
	is generally not true with finite constants $c_{1,p,\alpha}$ and $c_{2,p,\alpha}$ for c\`adl\`ag martingales without assuming predictability of $H$.
	To see this, let $q\in (0,1)$ and let $S_{q,\alpha}$ be a random variable such that
	\[S_{q,\alpha}=\begin{dcases}
	(1-q)^{1-\frac{1}{\alpha}}q^{-1}, &\text{ with probability } q;\\-(1-q)^{-\frac{1}{\alpha}}, &\text{ with probability } 1-q.
	\end{dcases}\] 
	Consider  $M_{q,\alpha}(t):=\mathbf{1}_{[1,\infty)}(t)S_{q,\alpha}$,  $H_{q,\alpha}(t):=\mathbf{1}_{[1,\infty)}(t)(S_{q,\alpha})_-$ (with $x_-:=(-x)\vee 0$, $x \in \mathbb{R}$)
	and $Z_{q,\alpha}(t):=M_{q,\alpha}(t)+H_{q,\alpha}(t)$. Then there is no constant $c_{p,\alpha}$ depending only on $p,\alpha \in (0,1)$ such that the inequality
	\[\mathbb{E}\left[(Z^*_{q,\alpha}(1))^p\right]\leq c_{p,\alpha}\left(\mathbb{E}[(H_{q,\alpha}(1))^\alpha]\right)^{p/\alpha}\]
	holds for all $q\in (0,1)$ since 
	\[\mathbb{E}\left[(Z^*_{q,\alpha}(1))^p\right]=\mathbb{E}\left[(S_{q,\alpha})_+^p\right]=(1-q)^{p(1-\frac{1}{\alpha})}q^{1-p}\to \infty,\quad  \text{ as } q\to 1,\]
	while, on the other hand,
	\[\mathbb{E}\left[(H_{q,\alpha}(1))^\alpha\right]=\mathbb{E}((S_{q,\alpha})_-^\alpha)=1.\]
\end{counterexample}

\section{Well-posedness of Path-dependent SDEs} 
First, we recall the definition of an orthogonal martingale-valued measure according to \cite{el1990martingale, walsh1986introduction}. Let  $(U,\mathcal{U})$ be a
Lusin space, i.e.~a measurable space homeomorphic to a Borel subset of $\R$. Consider an increasing sequence $U_n, n\in \mathbb{N}$ in $\mathcal{U}$
such that $U=\cup_{n\in \mathbb{N}}U_n$ and define $\mathcal{U}_n:=\mathcal{U}\vert_{U_n}$ and $ \mathcal{A}:=\cup_{n\in\mathbb{N}}\mathcal{U}_n$.
A {\em martingale measure} is a set function $\tilde{M}:\mathbb{R}^+\times \mathcal{A}\times \Omega\to \mathbb{R}$ which satisfies the following
(c.f. \cite{applebaum2006martingale, el1990martingale, walsh1986introduction}):
\begin{itemize}
\item[(a)] $\tilde{M}(0,A)=\tilde{M}(t,\emptyset)=0$ (a.s.), for all $A\in \mathcal{A}, t\geq 0$;
\item[(b)] $\tilde{M}(t,A\cup B)=\tilde{M}(t,A)+\tilde{M}(t,B)$ (a.s.), for all $t\geq 0$ and all disjoint $A,B\in \mathcal{A}$;
\item[(c)] For each non-increasing sequence $(A_i)$ of $\mathcal{U}_n$ converging to $\emptyset$, and for each $t \ge 0$,  $\mathbb{E}\left[\left\lvert \tilde{M}(t,A_i)\right\rvert^2\right]$
  tends to zero; 
\item[(d)] $\sup\left\lbrace \mathbb{E}\left\lvert\tilde{M}(t,A)\right\rvert^2, A\in \mathcal{U}_n\right\rbrace<\infty$ for all $n\in \mathbb{N}$ and $t\ge 0$;
\item[(e)] $(\tilde{M}(t,A))_{t\geq 0}$ is a c\`adl\`ag martingale for all $A\in\mathcal{A}$.

\end{itemize}
 Note that $\tilde{M}$ is countably additive on $\mathcal{U}_n$ as an $L^2$-valued set function.  In Walsh's terminology \cite{walsh1986introduction}, $\tilde{M}$ is called ``$\sigma$-finite $L^2$-valued martingale measure''.
 
 A martingale measure $\tilde{M}$ is called {\em orthogonal} if for all $A, B\in \mathcal{A}$ with $A\cap B=\emptyset$, $(\tilde{M}_t(A)\cdot \tilde{M}_t(B))_{t\geq 0}$ is a martingale.
 Note that in this case property (d) holds automatically.

 	Throughout the paper,  $\nu:\mathbb{R}^+\times \mathcal{U}\to \mathbb{R}\cup\{+\infty\}$ denotes a deterministic function such that for each $t\geq 0$, $\nu(t,\cdot)$ is a $\sigma$-finite  measure and the map $t\mapsto \nu(t,A)$ is measurable and locally integrable for each $A\in \mathcal{A}$. We assume that $\tilde{M}$ is an orthogonal martingale measure with intensity $(\nu_t)_{t\geq 0}$, i.e. $\left\langle \tilde{M}_\cdot(A), \tilde{M}_\cdot(B)\right\rangle_t=\int_0^t \nu_r(A\cap B)\,\de r$, which means $\big(\tilde{M}(t,A)\tilde{M}(t,B)-\int_0^t \nu_r(A\cap B)\,\de r\big)_{t\geq 0} $ is a martingale for all $A,B\in\mathcal{A}$. 

        The stochastic integral with respect to $\tilde{M}$ can be constructed in the same way as the construction of It\^o's integral (see \cite{walsh1986introduction}).
In particular, the stochastic integral $h\cdot \tilde M$ is defined for functions $h$ in
 \begin{align*}
 L^2_{\nu}:=\Big\lbrace h:(\mathbb{R}^+\times \Omega \times U, \mathcal{P}\otimes \mathcal{U})\to& (\mathbb{R}^{d},\mathcal{B}(\mathbb{R}^{d}));\\& \mathbb{E}\int_0^T\int_U  \left\lvert h(s,\omega, \xi)\right\rvert^2\nu_s(\de \xi)\de s<\infty, \forall T>0 \Big\rbrace,
 \end{align*}
 where $\mathcal{P}$ denotes the {\em predictable} $\sigma$-field on $\mathbb{R}^+\times \Omega$. Further, $h\cdot \tilde M$ is itself an orthogonal martingale measure and we have
 \begin{equation} \label{Chapter1-covariation}
 \left\langle h\cdot \tilde{M}_\cdot(A),h\cdot \tilde{M}_\cdot(B)\right\rangle_t=\int_0^t\left\lvert h(s,\omega,\xi)\right\rvert^2 \nu_s(A\cap B)\,\de s. 
 \end{equation}
Applying the usual localization procedure, the class of admissible integrands can be further extended to the class of measurable functions $h:(\mathbb{R}^+\times \Omega \times U, \mathcal{P}\otimes \mathcal{U})\to (\mathbb{R}^{d},\mathcal{B}(\mathbb{R}^{d}))$ for which $\int_0^T\int_U  \left\lvert h(s,\omega, \xi)\right\rvert^2\nu_s(\de \xi)\de s<\infty, \forall T>0$, almost surely.
In this case, \eqref{Chapter1-covariation} still holds.

\medskip

 Now we are ready to provide a general existence and uniqueness result on strong solutions of functional  stochastic differential equations with monotone coefficients driven by (orthogonal)
 martingale noise as above. 
\medskip

Consider the following path-dependent stochastic differential equation 
\begin{equation}
\label{Chapter2-equ1-}
\begin{dcases}
\de X_t 
= f(t,\omega,X)\,\de t
+ \int_U g(t,\omega, X,\xi)\tilde{M}(\de t,\de \xi),
\\ 
X_t  = z_t, \quad t\in [-\tau,0],
\end{dcases} 
\end{equation}  
where $\tau>0$ and the random initial condition $z$ belongs to
$\text{C\`adl\`ag}\,([-\tau,0];\mathbb{R}^d)$ and is $\mathcal{F}_0$ measurable. All spaces of c\`adl\`ag functions are endowed with the supremum norm. The coefficient
\[ \begin{aligned}
f & :\left( [0, \infty )\times \Omega\times \text{C\`adl\`ag}\,([-\tau,\infty);\mathbb{R}^d),\mathcal{BF}\otimes\mathcal{B}\left(\text{C\`adl\`ag}\,([-\tau,\infty);\mathbb{R}^d)\right)\right)\\&\quad \to \left(\mathbb{R}^d,\mathcal{B}\left(\mathbb{R}^d\right)\right)
\end{aligned}\]
is progressively measurable and 
\[\begin{aligned} g  &: \left([0, \infty )\times \Omega\times \text{C\`adl\`ag}\,([-\tau,\infty);\mathbb{R}^d)\times U,\mathcal{P}\otimes\mathcal{B}\left(\text{C\`adl\`ag}\,([-\tau,\infty);\mathbb{R}^d)\right)\otimes \mathcal{U}\right)\\&\quad\to\left(\mathbb{R}^{d},\mathcal{B}\left(\mathbb{R}^{d}\right)\right) \end{aligned}\]
is predictable. Here $\mathcal{BF}$ is the $\sigma$-field of progressively measurable sets on $[0,\infty)\times \Omega$. 
For every $t\in[0,\infty)$ and $\omega\in\Omega$, $f(t,\omega,x)$ depends only on the path of $x$ on the interval $[-\tau,t]$ and for every $t,\omega, \xi$, $g(t,\omega,x,\xi)$ depends only on the path of $x$ on the interval $[-\tau,t)$.


The following monotonicity and growth conditions are assumed:

\begin{hyp}
	\label{Chapter2-hyp1-} 
	There exist non-negative functions $t\mapsto K(t) $, $L_R(t)$ and $\tilde{K}_R(t)$,  for all $R > 0$ in $L^1_ {\mathrm {loc}} ( [0,\infty), \de t)$ such that for all
        $x,y\in \text{C\`adl\`ag}\,([-\tau,\infty),\mathbb{R}^d)$ and all $t\geq 0$,
	\begin{enumerate}[label=(C\theenumi)]
		\item \label{Chapter2-C1} for $\sup_{s\in[-\tau,t]}\left\lvert x(s)\right\rvert,\sup_{s\in[-\tau,t]}\left\lvert y(s)\right\rvert\le R$,
		\begin{align*}
		& 2\left\langle  x(t^-)-y(t^-),f(t,\omega,x)-f(t,\omega,y)\right\rangle 
		+ \int_{U}\left\lvert g(t,\omega,x,\xi)-g(t,\omega,y,
		\xi)\right\rvert^2\nu_t ( \de \xi) 
		\\&\le L_R(t)\sup_{s\in[-\tau,t]}\left\lvert x(s)-y(s)\right\rvert^2;
		\end{align*}
		\item \label{Chapter2-C2}
		$ 
		2\left\langle x(t^-),f(t,\omega,x)\right\rangle  
		+ \int_{U} \left\lvert g(t,\omega,x,\xi)\right\rvert^2\nu_t(\de \xi) 
		\le K(t)\left(1+\sup_{s\in[-\tau,t]}\left\lvert x(s)\right\rvert^2\right); $
		\item \label{Chapter2-C3} 
		$x\mapsto f(t,\omega,x)$  as a function from $\text{C\`adl\`ag}\,([-\tau,\infty);
		\mathbb{R}^d)$ to $\mathbb{R}^d$ is continuous;
		\item \label{Chapter2-C4}  for $\sup_{s\in[-\tau,t]}\left\lvert x(s)\right\rvert\leq R$,
		\[\left\lvert  
		f(t,\omega,x)\right\rvert  
		+\int_U \left\lvert g(t,\omega, 
		x,\xi)\right\rvert^2\nu_t(\de \xi)
		\le \tilde{K}_R(t);\]
		\item \label{Chapter2-C5}
		$\mathbb{E}\sup_{s\in[-\tau,0]}\left\lvert z(s)\right\rvert^2<\infty.$
	\end{enumerate} 
\end{hyp} 

We are going to prove existence and uniqueness of a strong solution using the Euler method. To 
this end let us introduce for $n\in\mathbb{N}$ and $k\in\mathbb{N}_0$ the Euler approximation  
\begin{equation}
\label{Chapter2-df Xn-}
\begin{split}
X^{(n)}_t 
& = X^{(n)}_{\frac{k}{n}}+\int_{\frac{k}{n}}^t f\left(s,\omega,  
X^{(n)}_{\cdot\wedge \frac{k}{n}}\right)\,\de s 
\\&\quad +\int_{(\frac{k}{n},t]\times U} g\left(s,\omega,   X^{(n)}_{\cdot\wedge \frac{k}{n}},
\xi\right)\tilde{M}(\de s,\de \xi), \quad t\in \left]\frac{k}{n},\frac{k+1}{n}\right],
\end{split}
\end{equation}
to the solution of \eqref{Chapter2-equ1-}. Let $\kappa(n,t):=\frac{k}{n}$ for 
$t\in \left]\frac{k}{n},\frac{k+1}{n}\right]$, $k\geq 0$ and  $\kappa(n,t):=t$ for $t\in [-\tau,0]$. The process $X^{(n)}$ can be 
constructed inductively as follows: $X^{(n)}_t := z_t$ for $t\in [-\tau,0]$, and given 
$X^{(n)}_t$ is defined for $t\le \frac{k}{n}$ we can extend $X^{(n)}_t$ for 
$t\in \left]\frac{k}{n},\frac{k+1}{n}\right]$ using \eqref{Chapter2-df Xn-}. Note that 
$X^{(n)}$, $t\ge -\tau$ is c\`adl\`ag, adapted, and that the  
stochastic integrals are well-defined.  

\begin{thm}
	\label{Chapter2-thmA2}
	Under Hypothesis \ref{Chapter2-hyp1-}, equation \eqref{Chapter2-equ1-} has a unique strong solution $X$, and 
	$X^{(n)}$ converges to $X$ locally uniformly in probability, i.e. for all $T>0$,  
	\[
	\lim_{n\to\infty}\mathbb{P}\left\lbrace \sup_{t\in[0,T]}\left\lvert X^{(n)}_t-X_t\right\rvert 
	> \varepsilon\right\rbrace=0\qquad\forall\,\varepsilon > 0\, . 
	\]
\end{thm}

\begin{proof}  Let us define the remainder 
	\[
	p^{(n)}_t=X^{(n)}_{\kappa(n,t)}-X^{(n)}_t, \quad t\in [-\tau,\infty)\, . 
	\]
	Then $p^{(n)}$ is adapted and $p^{(n)}\big((k/n)^+\big)=0$ for every $k \in \mathbb{N}_0$. Further,
	\begin{equation}  
	\label{Chapter2-Xn=phi(Xn+pn)-}
	\begin{split}
	X^{(n)}_t 
	& = z_0 +\int_0^t f\left(s,\omega, X^{(n)}+ \mathbf{1}_{(\kappa(n,s),\kappa(n,s)+1/n)}p^{(n)}\right)\,\de s\\&\quad +\int_0^t \int_U g\left(s,\omega, 
	X^{(n)}
	+ \mathbf{1}_{(\kappa(n,s),\kappa(n,s)+1/n)} p^{(n)},\xi\right)\tilde{M}(\de s,\de \xi)\, .
	\end{split}
	\end{equation}   
	Fix $T > 0$ and define the stopping times 
	\[
	\tau^{(n)}_R:=\Big( \inf\left\lbrace t\ge 0: \left\lvert X^{(n)}_t\right\rvert > 
	\frac{R}{3}\right\rbrace \wedge T \Big)\mathbf{1}_{\{R > 3\sup_{s\in[-\tau,0]}\left\lvert z(s)\right\rvert\}}
	\]
	for given $R > 0$. Then  
	\[
	\left\lvert p^{(n) }_{t}\right\rvert \le \frac{2R}{3}, \left
	\lvert X^{(n) }_{t}\right\rvert \le \frac{R}{3},\quad t\in(0,
	\tau^{(n) }_R). 
	\]
	For $R > 3\sup_{s\in[-\tau,0]}\left\lvert z(s)\right\rvert$ the above inequalities extend to all $t\in [-\tau , \tau_R^{(n)})$ and 
	$\tau_R^{(n)} > 0$ due to the right continuity of $X_t^{(n)}$.

	  
	We will prove the following properties which complete the 
	proof of existence on $[0, T]$, and hence on $[0,\infty)$, since $T$ was arbitrary.  
	\begin{enumerate}[label=(\roman{*})] 
		\item \label{Chapter2-i-} For every $t\ge 0$, $   \mathbf{1}_{(0,\tau^{(n) }_R)}(t)\sup_{u\in(\kappa(n,t),t]}\left\lvert p^{(n)}_u\right\rvert\to 0$ in probability as $n\to\infty$.
		\item \label{Chapter2-ii-} $\mathbb{E}\sup_{u\in [0,T]}\left\lvert  X^{(n) }_{u\wedge \tau^{(n) }_R}\right\rvert^{2p}\le C(T,R,n,p)$, for some $C(T,R,n,p)$ satisfying \[\lim_{n\to\infty}C(T,R,n,p)=\tilde{C}(T,p) \mbox{ for all } p \in (0,1),\, R>0.\]
		\item \label{Chapter2-iii-} $\lim_{R\to \infty}\limsup_{n\to\infty}\mathbb{P}\left\lbrace \tau^{(n) }_R < T\right\rbrace=0$.
		\item \label{Chapter2-iv-} $\forall \varepsilon > 0, \lim_{n,m\to\infty}\mathbb{P}\left\lbrace \sup_{t\in[0,T]}\left\lvert X^{(n) }_t-X^{(m) }_t\right\rvert >\varepsilon\right\rbrace=0$.
		\item \label{Chapter2-v-}  $\exists X: \forall \varepsilon>0, \lim_{n\to\infty}\mathbb{P}\left\lbrace \sup_{t\in[0,T]}\left\lvert X^{(n) }_t-X_t\right\rvert >\varepsilon\right\rbrace=0$ and $X$ is a strong solution of equation \eqref{Chapter2-equ1-} on $[0, T]$.
	\end{enumerate}
    \paragraph*{Proof of \ref{Chapter2-i-}:} Fix $t>0$ and $\varepsilon >0$. Using \eqref{Chapter2-df Xn-} and Hypothesis \ref{Chapter2-hyp1-} (C4), we have

		\begin{align*}
		\MoveEqLeft[0]\mathbb{P}\left\lbrace \sup_{u\in(\kappa(n,t),t]}\left\lvert p^{(n)}_u\right\rvert\ge \varepsilon,\tau^{(n)}_R>t\right\rbrace\\ 
                  & \le \mathbb{P}\left\lbrace \int_{\kappa(n,t)}^t \left\lvert f(s,\omega, X^{(n)}+\mathbf{1}_{(\kappa(n,s),\kappa(n,s)+\frac 1n)}p^{(n)})\right\rvert\, \de s\ge \varepsilon/2, \tau^{(n)}_R> t \right\rbrace \\
                    & +\mathbb{P}\left\lbrace \left\lvert\sup_{u\in(\kappa(n,t),t]}\int_{\kappa(n,t)}^u\int_U \mathbf{1}_{\{s\leq\tau^{(n)}_R 
			\}} g\left(s,\omega, X^{(n)}+\mathbf{1}_{(\kappa(n,s),s)}p^{(n)},\xi\right)\tilde{M}(\de s,\de \xi)\right\rvert\ge \varepsilon/2, \tau^{(n)}_R> t\right\rbrace \\ 
              & \le \mathbb{P}\left\lbrace\int_{\kappa(n,t)}^t \tilde{K}_R(s)\,\de s 
		\ge \varepsilon/2\right\rbrace \\&+\frac{4}{\varepsilon^2} \mathbb{E} \left(\sup_{u\in(\kappa(n,t),t]} \left\lvert \int_{\kappa(n,t)}^u 
		\int_U \mathbf{1}_{\left\lbrace s\leq\tau^{(n)}_R 
			\right\rbrace} g\left(s,\omega, X^{(n)}+\mathbf{1}_{(\kappa(n,s),s)}p^{(n)},\xi\right) 
		\tilde{M}(\de s,\de \xi)\right\rvert^2  \right)
		\end{align*}

                Using Burkholder-Davis-Gundy's inequality,  we continue as follows
	{\small
		\begin{align*}
		& \le \frac{2}{\varepsilon}\int_{\kappa(n,t)}^t \tilde{K}_R(s)\,\de s 
		+\frac{C}{\varepsilon^2}\mathbb{E} \left[\int_{\kappa(n,t)}^{\cdot}  \int_U \mathbf{1}_{\left\lbrace s\leq\tau^{(n)}_R 
			\right\rbrace} g\left(s,\omega, X^{(n)}+\mathbf{1}_{(\kappa(n,s),s)}p^{(n)},\xi\right) \tilde{M}(\de s,\de \xi)\right]_t
		\\ & \le \frac{2}{\varepsilon}\int_{\kappa(n,t)}^t \tilde{K}_R(s)\,\de s 
		+\frac{C}{\varepsilon^2}\mathbb{E} \int_{\kappa(n,t)}^{t}  \int_U \mathbf{1}_{\left\lbrace s\leq\tau^{(n)}_R 
			\right\rbrace} \left\lvert g\left(s,\omega, X^{(n)}+\mathbf{1}_{(\kappa(n,s),s)}p^{(n)},\xi\right) \right\rvert^2
		\nu_s(\de \xi)\de s\\& \le \frac{2}{\varepsilon}\int_{\kappa(n,t)}^t \tilde{K}_R(s)\,\de s 
		+\frac{C}{\varepsilon^2}\mathbb{E} \left( \int_{\kappa(n,t)}^{t}  \tilde{K}_R(s)\,\de s \right) 
		= \left(\frac{2}{\varepsilon} + \frac{C}{\varepsilon^2} \right) 
		\int_{\kappa(n,t)}^t \tilde{K}_R(s)\, \de s\, , 
		\end{align*}}
	so
	\[
	\limsup_{n\to \infty} 
	\mathbb{P}\left\lbrace \sup_{u\in(\kappa(n,t),t]}\left\lvert p^{(n)}_u\right\rvert\ge \varepsilon,  
	\tau^{(n)}_R> t\right\rbrace = 0
	\]
	which implies \ref{Chapter2-i-} since $\varepsilon>0$ was arbitrary.

        \paragraph*{Proof of \ref{Chapter2-ii-}:} Using It\^o's formula, we obtain 
	\begin{equation*} 
	\begin{aligned} 
	\left\lvert X_t^{(n)}\right\rvert^2 & = \left\lvert z_0\right\rvert^2
	+ \int_0^t 
	2\left\langle X^{(n)}_{s-},f\left(s, \omega, 
	X^{(n)}+\mathbf{1}_{(\kappa(n,s),s+\frac 1n)}p^{(n)}\right) \right\rangle \,\de s
	\\&\quad +\int_0^t\int_U \left\lvert  g\left(s,\omega,X^{(n)}+\mathbf{1}_{(\kappa(n,s),s)}p^{(n)},  \xi 
	\right)\right\rvert^2\nu_s(\de \xi)\,\de s
	+ M^{(n)}_t \, . 
	\end{aligned}
	\end{equation*}
	where 
	\begin{align*}
	M^{(n)}_t&:= \int_0^t\int_U 2\left\langle X^{(n)}_{s-}, g\left(s,\omega,X^{(n)}+\mathbf{1}_{(\kappa(n,s),s)}p^{(n)},  \xi 
	\right)\right\rangle \tilde{M}(\de s, \de \xi)\\&\quad +\left[\int_0^\cdot\int_U g\left(s,\omega,X^{(n)}+\mathbf{1}_{(\kappa(n,s),s)}p^{(n)},  \xi
	\right)\tilde{M}(\de s,\de \xi)\right]_t\\&\quad -\int_0^t\int_U \left\lvert g\left(s,\omega,X^{(n)}+\mathbf{1}_{(\kappa(n,s),s)}p^{(n)},  \xi
	\right)\right\rvert^2\nu_s(\de \xi)\de s
	\end{align*}
	and $\left(M^{(n)}_{t\wedge\tau_R^{(n)}}\right)_{t\geq 0}$ is a local martingale. 
        Using  (C2) and (C4), we have 
	\begin{equation*} 
	\begin{aligned}
	\left\lvert X^{(n)}_{t\wedge \tau_R^{(n)}}\right\rvert^2
	& \le \left\lvert z_0\right\rvert^2+\int_0^{t\wedge \tau_R^{(n)}} 2\left\langle X^{(n)}_{s^-}-X^{(n)}_{\kappa(n,s)},f\left(s, \omega,  X^{(n)}_{\cdot\wedge\kappa(n,s)} \right) \right\rangle \de s
	\\ 
	& \quad + \int_0^{t\wedge \tau_R^{(n)}}  K(s)
	\left(1+\sup_{u\in [-\tau,s]}\left\lvert X^{(n)}_u\right\rvert^2\right)\,\de s+M^{(n)}_{t\wedge\tau_R^{(n)}}\\
	& \le\left\lvert z_0\right\rvert^2+2 \int_0^t  \mathbf{1}_{\left\lbrace s\in (0, \tau_R^{(n)}]\right\rbrace} \tilde{K}_R(s)\left\lvert p^{(n)}_{s^-}\right\rvert \de s 
	\\ 
	& \quad + \int_0^t 
	K(s) \left[1+\sup_{u\in [-\tau,0]}\left\lvert z_u\right\rvert^2+\sup_{u\in [0,s]}\left\lvert X^{(n)}_{u\wedge \tau_R^{(n)}}\right\rvert^2\right]\de s +M^{(n)}_{t\wedge\tau_R^{(n)}}\\&= \int_{0}^{t} K(s)  \sup_{u\in [0,s]}\left\lvert X^{(n)}_{u\wedge \tau_R^{(n)}} \right\rvert^2\de s+H^{n,R}_t+M^{(n)}_{t\wedge\tau_R^{(n)}}.
	\end{aligned}
	\end{equation*}
	where
	\[H^{n,R}_t:= \left\lvert z_0\right\rvert^2+\int_{0}^{t}\left[K(s)\left(1+ \sup_{u\in [-\tau,0]}\left\lvert z_u\right\rvert^2\right)+2\cdot\mathbf{1}_{\left\lbrace s\in (0, \tau_R^{(n)}]\right\rbrace} \tilde{K}_R(s)\left\lvert p^{(n)}_{s^-}\right\rvert\right]\de s\, .\]
	Using Theorem \ref{Chapter2-GronwallThm}, we get for $p\in(0,1)$ that
	\[\mathbb{E}\left[\sup_{u\in [0,T]}\left\lvert X^{(n)}_{u\wedge \tau_R^{(n)}} \right\rvert^{2p}\right]\leq C(T,p)\left(\mathbb{E}H^{n,R}_T\right)^p=:C(T,R,n,p)\]
	where, by \ref{Chapter2-i-} and the dominated convergence theorem, 
	\[\lim_{n\to\infty}\mathbb{E}H^{n,R}_T= \mathbb{E} \left[\left\lvert z_0\right\rvert^2+\int_0^T K(s)\left(1+\sup_{u\in [-\tau,0]}\left\lvert z_u\right\rvert^2\right)\de s\right].\]
	Hence $\lim_{n\to\infty} C(T,R,n,p)=:\tilde{C}(T,p)$ exists and is independent of $R>0$ and hence \ref{Chapter2-ii-} holds. 

        \paragraph*{Proof of \ref{Chapter2-iii-}:} We have, for $p\in (0,1)$,
	\begin{align*}
	\MoveEqLeft[2]\limsup_{R\to\infty}\limsup_{n\to \infty}\,  
	\mathbb{P}\left\lbrace 
	\sup_{t\in[0,\tau^{(n)}_R]}\left\lvert X^{(n)}_t\right\rvert\ge \frac{R}{4};\tau^{(n)}_R 
	< T \right\rbrace 
	\\&\le \limsup_{R\to\infty}\limsup_{n\to \infty}\mathbb{P}\left\lbrace 
	\sup_{t\in[0,T\wedge\tau^{(n)}_R]}\left\lvert X^{(n)}_t\right\rvert \ge \frac{R}{4}
	\right\rbrace \\ 
	& \le\limsup_{R\to\infty}\left(\frac4R\right)^{2p}\limsup_{n\to\infty} 
	C(T,R,n,p) \leq \limsup_{R\to\infty}\left(\frac4R\right)^{2p}  \tilde{C}(T,p)=0.
	\end{align*} 
	It follows that  
	\begin{align*} 
	\limsup_{R\to\infty}\limsup_{n\to \infty}\mathbb{P}\left\lbrace \tau^{(n)}_R 
	< T\right\rbrace 
	\le\limsup_{R\to\infty}\limsup_{n\to \infty}\mathbb{P}\left\lbrace  
	\sup_{t\in[0,\tau^{(n)}_R]}\left\lvert X^{(n)}_t\right\rvert\ge \frac{R}{4};\tau^{(n)}_R 
	< T \right\rbrace=0
	\end{align*}
	which completes the proof of \ref{Chapter2-iii-}.  
	
	\paragraph*{Proof of \ref{Chapter2-iv-}:} Let $\tau^{n,m}_R:=
	\tau^{(n)}_R\wedge \tau^{(m)}_R$. Using It\^o's formula, we have
	\begin{align*}
	\MoveEqLeft[0]\left\lvert X^{(n)}_t- X^{(m)}_t\right\rvert^2 \\&= M^{n,m}_t+\int_0^t 2\Big\langle X^{(n)}_{s^-}-X^{(m)}_{s^-}, 
	f\left(s,\omega, X^{(n)}_{\cdot\wedge\kappa(n,s)}\right)-f\left(s,\omega, 
	X^{(m)}_{\cdot\wedge\kappa(m,s)}\right)\Big\rangle\de s\\&+ \int_0^t\int_U\left\lvert  g\left(s,\omega, 
	X^{(n)}_{\cdot\wedge\kappa(n,s)},\xi \right)-g\left(s,\omega, 
	X^{(m)}_{\cdot\wedge\kappa(m,s)},\xi\right)\right\rvert^2\nu_s(\de \xi)\de s
	\end{align*}
	where $\left(M^{n,m}_{t\wedge\tau^{n,m}_R} \right)_{t\geq 0}$ is a local  martingale starting from zero.
	Hypothesis (C1) 
        implies
	{\small
		\begin{align*}
		\MoveEqLeft[0]\left\lvert X^{(n)}_{t\wedge \tau^{n,m}_R}- X^{(m)}_{t\wedge \tau^{n,m}_R}\right\rvert^2 
		\\& \le 
		\int_0^{t\wedge \tau^{n,m}_R} 2\Big\langle X^{(n)}_{s^-}-X^{(m)}_{s^-}-X^{(n)}_{\kappa(n,s)}+X^{(m)}_{\kappa(m,s)}, 
		f\left(s,\omega, X^{(n)}_{\cdot\wedge\kappa(n,s)}\right) -f \left(s, \omega,X^{(m)}_{\cdot\wedge\kappa(m,s)}\right)\Big\rangle
		\de s \\ 
		& \quad+ \int_0^{t\wedge \tau^{n,m}_R} L_R (s) \sup_{u\in[-\tau,s]}\left\lvert  X^{(n)}_u+\mathbf{1}_{(\kappa(n,s),s]}(u)p^{(n)}_u-X^{(m)}_u-\mathbf{1}_{(\kappa(m,s),s]}p^{(m)}_u\right\rvert^2\de s+M^{n,m}_{t\wedge\tau^{n,m}_R}\\
		& \le  2\int_0^t \mathbf{1}_{(0,\tau^{n,m}_R)}(s)\Bigg\lbrace 
		\tilde{K}_R(s)\left\lvert p^{(n)}_{s^-}+p^{(m)}_{s^-}\right\rvert
		+ L_R(s)R\left(\sup_{u\in(\kappa(n,s),s]}\left\lvert p^{(n)}_{u}\right\rvert+\sup_{u\in(\kappa(m,s),s]}\left\lvert p^{(m)}_{u}\right\rvert\right)\Bigg\rbrace\de s   \\ 
		&\quad 
		+ \int_0^t L_R (s) \sup_{u\in[0,s]}\left\lvert X^{(n)}_{u\wedge\tau^{n,m}_R} 
		- X^{(m)}_{u\wedge\tau^{n,m}_R}\right\rvert^2\de s+M^{n,m}_{t\wedge\tau^{n,m}_R}
		\\&\leq \int_0^t L_R (s) \sup_{u\in[0,s]}\left\lvert X^{(n)}_{u\wedge\tau^{n,m}_R} 
		- X^{(m)}_{u\wedge\tau^{n,m}_R}\right\rvert^2\de s+H^{n,m,R}_t+M^{n,m}_{t\wedge\tau^{n,m}_R},
		\end{align*}}
	where 
	\begin{align*}
	H^{n,m,R}_t&:=2\int_0^t  \mathbf{1}_{(0,\tau^{n,m}_R)}(s)\Bigg\lbrace 
	\tilde{K}_R(s)\left\lvert p^{(n)}_{s^-}+p^{(m)}_{s^-}\right\rvert
	\\&\qquad\qquad\qquad\qquad\qquad+ L_R(s)R\left(\sup_{u\in(\kappa(n,s),s]}\left\lvert p^{(n)}_{u}\right\rvert+\sup_{u\in(\kappa(m,s),s]}\left\lvert p^{(m)}_{u}\right\rvert\right)\Bigg\rbrace\de s.
	\end{align*}
	Using Theorem \ref{Chapter2-GronwallThm}, we have for $p\in(0,1)$ that 
	\begin{equation}
	\mathbb{E}\left[\sup_{t\in[0,T]}\left\lvert X^{(n)}_{t\wedge\tau^{n,m}_R} 
	- X^{(m) }_{t\wedge\tau^{n,m}_R}\right\rvert^{2p}\right]\le C(T,R,p)\left(\mathbb{E}H^{n,m,R}_T\right)^p.
	\end{equation}
	Hence for $a>0$,
	\begin{align*}
	\mathbb{P} & \left\lbrace \sup_{t\in[0,T]}\left\lvert X^{(n)}_t-X^{(m)}_t\right\rvert\ge 
	a\right\rbrace \\ 
	& \qquad \le \mathbb{P}\left\lbrace T>\tau^{(n)}_R\right\rbrace 
	+\mathbb{P}\left\lbrace T>\tau^{(m)}_R\right\rbrace +
	\mathbb{P}\left\lbrace \sup_{t\in[0,\tau^{n,m}_R]}\left\lvert X^{(n)}_t-X^{(m)}_t\right\rvert\ge 
	a\right\rbrace \\ 
	& \qquad \le\mathbb{P}\left\lbrace T>\tau^{(n)}_R\right\rbrace 
	+ \mathbb{P}\left\lbrace T>\tau^{(m)}_R\right\rbrace 
	+ \frac{1}{a^{2p}}\mathbb{E}\left[\sup_{t\in[0,T]}\left\lvert X^{(n)}_{t\wedge\tau^{n,m}_R} 
	- X^{(m) }_{t\wedge\tau^{n,m}_R}\right\rvert^{2p}\right] \\ 
	& \qquad \le \mathbb{P}\left\lbrace T>\tau^{(n)}_R\right\rbrace 
	+ \mathbb{P}\left\lbrace T>\tau^{(m)}_R\right\rbrace 
	+a^{-2p}C(T,R,p)\left(\mathbb{E}H^{n,m,R}_T\right)^p \, . 
	\end{align*}
	\ref{Chapter2-i-} and dominated convergence now imply that 
	\[
	\limsup_{n,m\to \infty}\mathbb{E}H^{n,m,R}_T=0 
	\]
	and using \ref{Chapter2-iii-}, we get 
	\begin{align*}
	\MoveEqLeft[1]\limsup_{n,m\to\infty}\mathbb{P} 
	\left\lbrace \sup_{t\in[0,T]}\left\lvert X^{(n) }_t-X^{(m) }_t\right\rvert\ge 
	a\right\rbrace \\ 
	& 
	\le  \lim_{R\to\infty}\limsup_{n,m\to\infty} 
	\left[\mathbb{P}\left\lbrace T>\tau^{(n) }_R\right\rbrace 
	+ \mathbb{P}\left\lbrace T > \tau^{(m) }_R\right\rbrace+a^{-2p}C(T,R,p)\left(\mathbb{E}H^{n,m,R}_T\right)^p\right]=0,
	\end{align*}
	so \ref{Chapter2-iv-} is obtained.

	\paragraph*{Proof of \ref{Chapter2-v-}:} Since the space $ 
	\text{C\`adl\`ag}\,([-\tau,T],\mathbb{R}^d)$ is complete, via the Borel-Cantelli lemma, \ref{Chapter2-iv-} yields that there exists an adapted c\`adl\`ag process $X$ such that
	\[ 
	\lim_{n\to\infty}\mathbb{P}\left\lbrace \sup_{t\in[0,T]}\left\lvert X^{(n)}_t 
	- X_t\right\rvert\ge \varepsilon\right\rbrace=0. 
	\]
	We have to show that, for a subsequence of 
	$n\in\mathbb{N}$, all terms of equation \eqref{Chapter2-Xn=phi(Xn+pn)-} converge almost surely to the corresponding
        terms of equation \eqref{Chapter2-equ1-}. We have 
	\begin{align*}
	\MoveEqLeft[2]\limsup_{n\to\infty}\int_0^T\mathbb{P} \left\lbrace \sup_{u\in[0,t]}\left\lvert X^{(n)}_{u\wedge\kappa(n,t)} 
	-X_{u}\right\rvert\ge \varepsilon\right\rbrace \de t \\
	& \le\limsup_{n\to\infty}\int_0^T\mathbb{P}\left\lbrace \sup_{u\in[0,T]}\left\lvert X^{(n)}_u 
	- X_u\right\rvert\ge \varepsilon\right\rbrace \de t+ 
	\\&\quad+\limsup_{n\to\infty}\mathbb{E}\int_0^T\mathbf{1}_{\left\lbrace \sup_{u\in(\kappa(n,t),t]}\left\lvert 
	X_{\kappa(n,t)}-X_{u}\right\rvert\ge \varepsilon\right\rbrace}\de t=0\, . 
	\end{align*} 
	So we can find a subsequence, say 
	$\left\lbrace n_l\right\rbrace_{l\in\mathbb{N}}$, such that as $l\to\infty$,
	\[\sup_{u\in[0,t]}\left\lvert X^{(n_l)}_{u\wedge\kappa(n_l,t)} 
	-X_{u}\right\rvert\to 0 
	\quad \de t\otimes\mathbb{P}\text{-a.e.}\,  (t,\omega)\in [0,T]\times \Omega.
	\]
	Now let us define
	\[
	S(T):=\sup_{l\in\mathbb{N}} \sup_{t\in[0,T]}\sup_{u\in[0,t]}\left\lvert X^{(n_l) }_{u\wedge\kappa(n_l,t)}\right\rvert \, , 
	\]
	then
	\[
	S(T)<\infty\quad \mathbb{P}\text{-a.s.} 
	\]
	Therefore, using \ref{Chapter2-C3}, \ref{Chapter2-C4} and dominated convergence, we obtain that
	\[
	\lim_{l\to\infty}\int_0^t f\left( s,\omega,  X^{(n_l) }_{\cdot\wedge\kappa(n_l,s)}\right)\de s 
	= \int_0^t f\left(s,\omega,X\right)\de s\quad \mathbb{P}\text{-a.s.} 
	\]
	Let $\tau (R):=\inf\left\lbrace t\ge 0:S(t)>R\right\rbrace\wedge T$. Fix $t\in [0,T]$.  
	By (C1) and dominated convergence and 
	\begin{align*}
	\MoveEqLeft[3]\lim_{l\to\infty}\mathbb{E}\left\lvert \int_0^{t\wedge\tau (R)}\int_U\left[g\left(s,
	\omega,X^{(n_l) }_{\cdot\wedge\kappa(n_l,s)},\xi\right)-g\left( s,\omega, 
	X,\xi\right)\right]\tilde{M}(\de s,\de \xi)\right\rvert^2 \\ 
	& =\lim_{l\to\infty}\mathbb{E}\int_0^t\int_U \mathbf{1}_{\left\lbrace s\le \tau (R) 
		\right\rbrace}\left\lvert g\left(s,\omega,X^{(n_l) }_{\cdot\wedge\kappa(n_l,s)},\xi\right) 
	-g\left( s,\omega,X,\xi\right)\right\rvert^2\nu_s(\de \xi)\de s=0,
	\end{align*}
	so, for $t \in [0,T]$,
	\begin{align*}
	\MoveEqLeft[3]\mathbb{P}\left\lbrace \left\lvert \int_0^t\int_U\left[g\left(s,
	\omega,X^{(n_l) }_{\cdot\wedge\kappa(n_l,s)},\xi\right)-g\left( s,\omega, 
	X,\xi\right)\right]\tilde{M}(\de s,\de \xi)\right\rvert>\varepsilon\right\rbrace \\ 
	& \le \mathbb{P}\left\lbrace \left\lvert \int_0^{t\wedge\tau (R)}\int_U\left[g\left(s,
	\omega,X^{(n_l) }_{\cdot\wedge\kappa(n_l,s)},\xi\right)-g\left( s,\omega, 
	X,\xi\right)\right]\tilde{M}(\de s,\de \xi)\right\rvert>\varepsilon\right\rbrace\\&\quad+\mathbb{P}\left\lbrace t>\tau (R) 
	\right\rbrace.
	\end{align*}
	Fix some sufficiently large $R$ such that the second term on the right hand side is less than 
	$\delta>0$, then taking the limit $l\to\infty$ implies
	\[ 
	\lim_{l\to \infty}\mathbb{P} 
	\left\lbrace \left\lvert \int_0^t\int_U\left[g\left(s,
	\omega,X^{(n_l) }_{\cdot\wedge\kappa(n_l,s)},\xi\right)-g\left( s,\omega, 
	X,\xi\right)\right]\tilde{M}(\de s,\de \xi)\right\rvert > 
	\varepsilon\right\rbrace\le \delta\]
	where  $\delta>0$ is arbitrary. Therefore 
	\[ 
	\int_0^t\int_U g\left(s,
	\omega,X^{(n_l) }_{\cdot\wedge\kappa(n_l,s)},\xi\right)\tilde{M}(\de s,\de \xi)\to \int_0^t\int_U g\left( s,\omega, 
	X,\xi\right)\tilde{M}(\de s,\de \xi)\quad \text {in probability} 
	\]
	and for some subsequence $n_{l_k}$ the above convergence is $\mathbb{P}-a.s$. 
	Therefore $X$ is a solution of equation \eqref{Chapter2-equ1-} on $[0, T]$.
	\paragraph*{Uniqueness:} Let $X$ and $Y$ be two solutions of equation \eqref{Chapter2-equ1-} and define 
	\[
	\tau(R):=\inf\left\lbrace t\ge 0; \left\lvert X_t\right\rvert>R \text{  or  } \left\lvert 
	Y_t\right\rvert>R\right\rbrace \, . 
	\]
	Then 
	\begin{align*}
	\left\lvert X_{t\wedge \tau(R)}-Y_{t\wedge \tau(R)}\right\rvert^2 & = \int_0^{t\wedge \tau(R)} \Big\lbrace 2\left\langle X_{s^-}-Y_{s^-}, 
	f\left( s, \omega,X\right)-f\left(s, \omega , Y \right)\right\rangle\\ 
	& \quad + \int_U \left\lvert g\left( s,\omega , 
	X,\xi\right)-g\left(s,\omega,Y,\xi\right)\right\rvert^2\nu_s(\de \xi)\Big\rbrace\de s+M_{t\wedge\tau(R)} \\ 
	& \le \int_0^t \mathbf{1}_{\{s\leq \tau(R)\}}L_R(s)\sup_{u\in[0,s]}  
	\left\lvert X_{u}-Y_{u}\right\rvert^2\de s+M_{t\wedge\tau(R)}
	\\ & \le \int_0^t L_R(s)\sup_{u\in[0,s]}  
	\left\lvert X_{u\wedge \tau(R)}-Y_{u\wedge \tau(R)}\right\rvert^2\de s+M_{t\wedge\tau(R)}
	\end{align*} 
	where $\left(M_{t\wedge\tau(R)}\right)_{t\geq 0}$ is a local martingale starting from zero. Using Theorem \ref{Chapter2-GronwallThm}, for $p\in(0,1)$ we have 
	\begin{align*}
	\mathbb{E}\left[\sup_{s\in[0,T]}\left\lvert X_{s\wedge \tau(R)}-Y_{s\wedge \tau(R)}\right\rvert^{2p}\right]\leq 0. 
	\end{align*}
	Therefore $X_{s\wedge \tau(R)}-Y_{s\wedge \tau(R)}=0$ a.s.~and uniqueness is proved.
\end{proof}
\bibliographystyle{plain}
\nocite{prevot2007concise}
\nocite{gyongy2013note}
\nocite{kumar2014strong}

\bibliography{biblio}

\end{document}